\documentclass[a4paper,11pt]{article}
\usepackage{fullpage}
\usepackage{amsmath,amssymb,latexsym,amsthm}
\usepackage[pdftex]{graphicx}
\usepackage{graphics}
\usepackage{psfrag}
\usepackage{color, epsfig}
\usepackage{mathrsfs}
\usepackage[english]{babel}
\usepackage{bbm}
\newtheorem{theorem}{Theorem}[section]

\newtheorem{proposition}[theorem]{Proposition}

\newtheorem{remark}[theorem]{Remark}

\newtheorem{goal}[theorem]{Goal}

\numberwithin{equation}{section}

\title{Mixing Douglas' and weak majorization and factorization theorems}

\author{
Pierre Lissy\footnote{CERMICS, Ecole des Ponts, IP Paris, Marne-la-Vall\'ee, France.
}
}

\date\today
\begin{document}
\maketitle
\begin{abstract}The Douglas' majorization and factorization theorem characterizes the inclusion of operator ranges in Hilbert spaces. Notably, it reinforces the well-established connections between the inclusion of kernels of operators in Hilbert spaces and the (inverse) inclusion of the closures of the ranges of their adjoints. This note aims to present a ``mixed'' version of these concepts for operators with a codomain in a product space. Additionally, an application in control theory of coupled systems of linear  partial differential equations is presented.

\end{abstract}

\vspace*{0.3cm}


\vspace*{0.3cm}

\section{Introduction}

The goal of this note is to provide new insights into the celebrated Douglas' majorization and factorization theorem  proved in \cite{MR203464} (see Theorem \ref{th1}) on the inclusion of operator ranges in Hilbert spaces. We will reinterpret Douglas' majorization and factorization theorem as a reinforcement of well-known properties related to the inclusions of the closures of operator ranges. This can also be expressed in a form resembling a ``non-quantitative majorization'' and ``non-continuous factorization'' theorem (see Theorem \ref{th2}), which we will refer to as the ``weak majorization and factorization theorem''. We do not claim originality for this ``theorem''; rather, its significance lies in its analogy with Douglas' theorem.

Once this basic analogy is established, we will examine operators whose codomain is a product Hilbert space. This will enable us to derive a ``mixed'' version of the Douglas' and weak majorization and factorization theorems (our main Theorem \ref{th:main}). The necessary and sufficient conditions given in this theorem are not really easy to manipulate, so we give some sufficient conditions (Propositions \ref{prop1} and \ref{propp}) that might be easier to verify in practice.

It is well-known that Douglas' majorization and factorization theorem is an essential tool for studying abstract linear control systems in infinite dimensional spaces (see the seminal paper \cite{77DR}). Here, the applications we have in mind come from the control theory of coupled systems of linear  partial differential equations (see \textit{e.g.} \cite[Theorem 1.10]{MR4482305}, where the kind of mixed version presented later on appears naturally), but our hope is that the present note might also be helpful in other fields of mathematics. Notably, we present in Section \ref{app} an application to a toy model of coupled heat system. This is for illustration purpose and  we do not intend to transform this note into a specialized paper in control theory. Further advanced applications in control theory  will be discussed elsewhere.

\subsection{Douglas' majorization and factorization Theorem and a weaker analogue}

Let us set some useful and standard notations. For a set $E$ in a metric space, $\overline E$ is the closure of $E$. For an operator $T$ between linear spaces, $\mathcal R(T)$ is the range of $T$, and $\mathcal N(T)$ is the kernel of $T$. For $E$ and $F$ two normed vector spaces, $\mathcal  L_c(E,F)$ is the set of linear continuous maps from $E$ to $F$ (endowed with the operator norm $|||\cdot|||$), whereas $\mathcal  L(E,F)$ is the set of linear (and non-necessarily continuous) maps from $E$ to $F$.

\vspace*{0.1cm}

Now, let us consider $H_1,H_2,H_3$ some (all) real or (all) complex Hilbert space, endowed with some scalar or hermitian product (assumed to be left-linear in the complex case), with associated norm respectively $||\cdot||_i$ ($i=1,2,3$).  
 Consider two operators 
$A\in \mathcal L_c(H_1,H_3)$ and $B\in \mathcal L_c(H_2,H_3)$. Then, we have the following so-called Douglas' majorization and factorization theorem: 

\begin{theorem}\label{th1}
The  following statements are equivalent :
\begin{itemize}
\item[(i)] (range inclusion) $\mathcal R(A)\subset \mathcal R(B)$ (or, equivalently, for any $h_1\in H_1$, there exists $h_2\in H_2$  such that $Ah_1=Bh_2$).
\item[(ii)] (majorization) There exists $C>0$ such that for any $ z\in H_3$, we have 
$$||A^*z||_1\leqslant C ||B^*z||_2.$$
\item[(iii)] (factorization) $A= BC$ for some $C\in \mathcal L_c(H_1,H_2)$ (or, equivalently, $A^*=DB^*$ for some $D\in \mathcal L_c(H_2,H_1)$).
\end{itemize}
\end{theorem}

Theorem \ref{th1} can be seen as a refinement of the following ``theorem''.

\begin{theorem}\label{th2}
The  following statements are equivalent :
\begin{itemize}
\item[(i)] (closure of range inclusion)  $\overline{\mathcal R(A)}\subset \overline{\mathcal R(B)}$  (or, equivalently, ${\mathcal R(A)}\subset \overline{\mathcal R(B)}$, or equivalently, for any $h_1\in H_1$ and any $\varepsilon>0$, there exists $h_2\in H_2$ such that $||Ah_1-Bh_2||_3\leqslant \varepsilon$).
\item[(ii)] (non-quantitative majorization)  $\mathcal N(B^*)\subset \mathcal N(A^*)$, \textit{i.e.} $$\forall  z\in H_3,\, B^*z=0 \Rightarrow A^*z=0.$$
\item[(iii)] (non-continous factorization) $A^*=DB^*$ for some $D\in \mathcal L(H_2,H_1)$ not necessarily continuous.
\end{itemize}
\end{theorem}
\begin{remark}Since $D$ needs not to be continuous here, the factorization property $A=BC$ for some $C\in \mathcal L(H_1,H_2)$ not necessarily continuous cannot be derived by passing to the adjoint. This is reasonable, since $A=BC$ implies that $\mathcal R(A)\subset \mathcal R(B)$, even if $C$ is not continuous, and this property is clearly stronger than the desired one  ${\mathcal R(A)}\subset \overline{\mathcal R(B)}$ in infinite dimension (consider for instance a non-onto operator $B$ with dense range, and any $A$ such that $\mathcal R(A)\not\subset\mathcal R(B)$).  Notably, by Theorem \ref{th1}, we see that $A=BC_1$ for some $C_1\in \mathcal L(H_1,H_2)$ not necessarily continuous implies that  $A=BC_2$ for some $C_2\in \mathcal L_c(H_1,H_2)$.
\end{remark}

\begin{proof}
We claim no originality in the statement  and proof of this ``theorem''. The equivalence between (i) and (ii) is  of course totally standard  by passing to the orthogonal. Let us remind how to prove the equivalence between (ii) and (iii).
That (iii) implies (ii) is trivial by linearity of $D$. The fact that (ii) implies (iii) is standard: we first set 
$$D : B^*h_2 \in \mathcal R(B^*) \mapsto A^*h_2 \in \mathcal R(A^*),$$
which is well-defined by (ii) and linear by easy computations. Then, we consider any algebraic complement $F$ of the range of $B^*$, and for $h_3= f+B^*h_2$ with $f\in F$, we set 
$D(h_3)=A^*h_2.$
Then $D$  is well-defined by (ii), linear by easy computations, and by setting $f=0$ we indeed have that for any $h_2\in H_2$, $D(B^*h_2)=A^*h_2$.
\end{proof}
\subsection{A mixed version for a range in a product space}

Now, assume that we have $H_3= H_4\times H_5$, where $H_4,H_5$  are some Hilbert spaces on the same field as $H_1,H_2$, endowed with some scalar or hermitian product, with associated norm respectively $||\cdot||_i$ ($i=4,5$). As usual, we endow in this case $H_3$ with the Hilbertian norm 
$$||(h_4,h_5)||_{H_3}^2=||h_4||^2_{H_4}+||h_5||^2_{H_5},\,\, (h_4,h_5)\in H_4\times H_5.$$
Then, for $h_i \in H_i$ ($i=1,\dots 4$), one can write 
$$A(h_1)=\begin{pmatrix}A_1(h_1)\\A_2(h_1)\end{pmatrix}= \begin{pmatrix}A_1\\A_2\end{pmatrix}h_1,$$
with $A_1\in \mathcal L_c(H_1,H_4),\,A_2\in\mathcal L_c(H_1,H_5)$,
or equivalently 
$$A^*(h_4,h_5)= A^*_1(h_4)+A^*_2(h_5)=\begin{pmatrix}A_1^*&A_2^*\end{pmatrix}\begin{pmatrix}h_4\\h_5\end{pmatrix},$$
and the same for $B$: 
$$B(h_2)=\begin{pmatrix}B_1(h_2)\\B_2(h_2)\end{pmatrix}= \begin{pmatrix}B_1\\B_2\end{pmatrix}h_2,$$
with $B_1\in \mathcal L_c(H_2,H_4),\,B_2\in\mathcal L_c(H_2,H_5)$, or equivalently 
$$B^*(h_4,h_5)= B^*_1(h_4)+B^*_2(h_5)=\begin{pmatrix}B_1^*&B_2^*\end{pmatrix}\begin{pmatrix}h_4\\h_5\end{pmatrix}.$$

Our main result will some kind of ``mixed'' usual-weak Douglas' majorization and factorization Theorem: for the first component of $A$, we ask for an exact inclusion of range, but for the second component, we ask for some approximate inclusion of range. More precisely,  our objective is to characterize the following property.

\begin{goal}\label{goal}
For any $h_1\in H_1$ and any $\varepsilon>0$, find  $h_2\in H_2$ such that $A_1h_1=B_1h_2$ and $||A_2h_1-B_2h_2||_5\leqslant \varepsilon$. 
\end{goal}

The important point here is that the same $h_2$  has to realise both objectives. Clearly, this has to be more restrictive than the following properties taken separately : 
\begin{itemize}\item For any $h_1\in H_1$, find  $h_2\in H_2$ such that $A_1h_1=B_1h_2$.
\item For any $h_1\in H_1$  and any $\varepsilon>0$,  find $\tilde h_2 \in H_2$ such that   $||A_1h_1-B_1\tilde h_2||_5\leqslant \varepsilon$ and $||A_2h_1-B_2\tilde h_2||_5\leqslant \varepsilon$. 
\end{itemize}
This gives some natural necessary conditions given in Proposition \ref{nec}. One of the goal of this note will be to understand how to fill the gap in order to find reinforcements of these necessary conditions that turn out to be also sufficient (see Theorem \ref{th:main}). %We will notably give two sufficient conditions, in Propositions \ref{prop1} and  \ref{propp}, that might be easier to verify in practice (see Section \ref{app}).

\begin{proposition}\label{nec}
Necessary conditions for Goal \ref{goal} to hold are the following (equivalent) properties:
(i) There exists $C>0$ such that for any $ h_4\in H_4$, we have 
$$||A^*_1h_4||_1\leqslant C ||B^*_1h_4||_2, $$ and $\mathcal N(B^*)\subset \mathcal N(A^*)$, \textit{i.e.} $\forall  (h_4,h_5)\in H_3,\, B_1^*h_4+B_2^*h_5=0 \Rightarrow A_1^*h_4+A_2^*h_5=0.$

\vspace*{0.2cm}

\noindent(ii)   $A_1= B_1C_1$ for some $C\in \mathcal L_c(H_1,H_4)$ (or, equivalently, $A^*_1=D_1B^*_1$ for some $D_1\in \mathcal L_c(H_4,H_1)$), and $A^*=DB^*$ for some $D\in \mathcal L(H_2,H_1)$ not necessarily continuous.

 \end{proposition}
 \begin{remark}\label{rerem}In general, properties (i) and (ii) of Proposition \ref{nec} are not sufficient for Goal \ref{goal} to hold. We will provide a counterexample in Proposition \ref{remex}.
 \end{remark}
 \begin{proof}Necessarily, if Goal \ref{goal} holds, then,   for any $h_1\in H_1$, there exists $h_2\in H_2$ such that $A_1h_1=B_1h_2$. So, one can apply Douglas' majorization and factorization theorem \ref{th1} and deduce that there exists $C>0$ such that for any $ z\in H_3$, we have 
$$||A^*_1z||_1\leqslant C ||B^*_1z||_2, $$ and $A_1= B_1C_1$ for some $C\in \mathcal L_c(H_1,H_4)$, or, equivalently, $A^*_1=D_1B^*_1$ for some $D_1\in \mathcal L_c(H_4,H_1)$.
 
Moreover, clearly, we also have that there exists   $h_2\in H_2$ such that $||A_1h_1-B_1h_2||^2_4+||A_2h_1-B_2h_2||^2 \leqslant \varepsilon$.
So, applying Theorem \ref{th2} leads to the $\mathcal N(B^*)\subset \mathcal N(A^*)$, \textit{i.e.} $\forall  (h_4,h_5)\in H_3,\, B_1^*h_4+B_2^*h_5=0 \Rightarrow A_1^*h_4+A_2^*h_5=0,$ and $A^*=DB^*$ for some $D\in \mathcal L(H_2,H_1)$ not necessarily continuous. This gives the (equivalent) necessary conditions (i) and (ii).\end{proof}
\section{Main result}

We can now state our mixed version of a majorization and factorization theorem.

\begin{theorem}\label{th:main}The following statements are equivalent:
\begin{itemize}
\item[(i)] (mixed range inclusion and closure of range  inclusion) Goal \ref{goal} is fullfilled: for  any $h_1\in H_1$ and any $\varepsilon>0$, there exists  $h_2\in H_2$ such that $A_1h_1=B_1h_2$ and $||A_2h_1-B_2h_2||\leqslant \varepsilon$.
\item[(ii)] (mixed quantitative and non-quantitative majorization)  The two following conditions hold.
\begin{itemize}\item There  exists $C>0$ such that for any  $ h_4\in H_4$, we have 
$$||A^*_1h_4||_1\leqslant C ||B^*_1h_4||_2.$$
\item  Consider any $h_5\in H_5$ such that there exists some $(y_n)\in H_4^{\mathbb N*}$ for which
 $B_2^*h_5=\lim_{n\rightarrow +\infty}  B_1^* y_n$. Then,  $A_2^*h_5=\lim_{n\rightarrow +\infty}  A_1^* y_n.$
\item[(iii)] (mixed continuous and non-continuous factorization)   Let us call $\Pi$ the orthogonal projection on $\mathcal N(B_1)$ in $H_2$.  Then, $A^*_1=D_1B^*_1$ and $A^*_2=(D_1+D_2\Pi)B^*_2$ for some $D_1\in \mathcal L_c(H_2,H_1)$ and $D_2\in \mathcal L(\mathcal N(B_1),H_1)$ not necessarily continuous.
\end{itemize}
\end{itemize}
\end{theorem}
\begin{remark} (ii) and (iii) enable to understand the gap between the necessary conditions given in Proposition \ref{nec}  and necessary and sufficient conditions. Notably, one readily observes that condition (ii) in Theorem \ref{th:main} is stronger than condition (i) in Proposition  \ref{nec} (consider a constant sequence), and  that condition (iii) in Theorem \ref{th:main} is stronger than condition (ii) in Proposition  \ref{nec} (just extend $D_2$ to $0$ on $\mathcal N(B_1)^\perp$ so that it is now defined on $H_2$).
\end{remark}
\begin{proof}

 Let us first prove that (i) implies (iii). 
First of all, we remark that (i) notably implies that $\mathcal R(A_1)\subset \mathcal R(B_1)$, so we can apply Theorem \ref{th1} and deduce that   $A_1= B_1C_1$  for some $C_1 \in \mathcal L_c(H_1,H_2)$.
By assumption, for any $\varepsilon>0$ and any $h_1\in H_1$, there exists  $h_2\in H_2$ such that $A_1h_1=B_1h_2$ and $||A_2h_1-B_2h_2||\leqslant \varepsilon$. For such $\varepsilon,h_1,h_2$, we have 
$B_1(C_1h_1-h_2)=0$, so $h_2= C_1h_1 +u$ for some $u\in\mathcal N(B_1)$. Now, for fixed $h_1$, (i) implies that we should also have 
$A_2h_1=B_2h_2=B_2C_1h_1+B_2u$, \textit{i.e.} 
$||A_2h_1-B_2C_1h_1-B_2u||\leqslant \varepsilon$ for some  $u\in \mathcal N(B_1)$. We deduce that 
$$\overline{\mathcal R(A_2-B_2C_1)}\subset\overline {\mathcal R (\tilde B_2)},$$ where $\tilde B_2=( B_2)_{|\mathcal N(B_1)}.$
Applying Theorem \ref{th2} implies that 
that there exists some $D_2\in \mathcal L(\mathcal N(B_1),H_1)$ not necessarily continuous such that 
$$A_2^*-C_1^*B_2^*=D_2\tilde B_2^*.$$
Remarking that $\tilde B_2^*= \Pi B_2^*$ and introducing $D_1=C_1^*$ gives that (iii) is verified.

\vspace*{0.2cm}

Now, let us prove that  (iii) implies (ii).  This is just a question of remarking that  for any $(h_4,h_5)\in H_4\times H_5$, we have 
$$\begin{aligned}A^*1h_4+A_2^*h_5&=D_1B_1^*h_4+(D_1+D_2\Pi)B^*_2h_5\\&=(D_1+D_2\Pi)(B_1^*h_4+B^*_2h_5),\end{aligned}$$
since $B_1\Pi=0$ by definition of $\Pi$ (so $\Pi^* B_1^*=\Pi B_1^*=0$). Choosing $h_5=0$ leads to  
$$||A^*_1h_4||_1\leqslant C ||B^*_1h_4||_2,$$
with $C=|||D_1|||$, which gives the first property of (ii). For the second property, it is a little bit more intricate. Assume that for some $h_5\in H_5$, we have some $(y_n)\in H_4^{\mathbb N*}$ such that 
 $$B_2^*h_5=\lim_{n\rightarrow +\infty}  B_1^* y_n.$$
This assertion  is equivalent to  the fact that  $B_2^*h_5\in \overline{\mathcal R(B^*_1)}= \mathcal N(B_1)^\perp$, \textit{i.e.}, is equivalent to the fact that $\Pi B_2^*h_5=0$. So, for such a $h_5$, we have by (iii) that 
$A^*_1h_5=D_1B^*_1h_5$ and $A^*_2h_5=D_1B^*_2h_5$. Now, remark that for any $n\in\mathbb N^*$, we have
$D_1B^*_1y_n=A_1y_n$. By continuity of $D_1$, this quantity converges  as $n\rightarrow +\infty$ to $D_1B^*_2h_5=A_2^*h_5$, whence the desired result.

\vspace*{0.2cm}

 It remains to prove that (ii) implies (i). First of all, applying Douglas' majorization and factorization theorem  \ref{th1} to $A_1$ and $B_1$, we know that $A_1=B_1C_1$ for some $C_1\in \mathcal L_c(H_1,H_2)$. Let us prove that 
 $${\mathcal R(A_2-B_2C_1)}\subset\overline{\mathcal R (\tilde B_2)},$$ 
 where $\tilde B_2=( B_2)_{|\mathcal N(B_1)}$, \textit{i.e.}, by passing to the orthogonal, 
 $$\mathcal N(\Pi B_2^*)\subset \mathcal N(A_2^*-C_1^*B_2^*).$$
 Let us consider some $h_4\in  \mathcal N(\Pi B_2^*)$. This means that $\Pi B_2^*h_4=0$, \textit{i.e.} $B_2^*h\in \mathcal N(B_1)^\perp$, \textit{i.e.} 
 $B_2^*h \in \overline{\mathcal R(B_1^*)}$. 
Hence, we  have some $(y_n)\in H_4^{\mathbb N*}$ such that 
 $$B_2^*h_5=\lim_{n\rightarrow +\infty}  B_1^* y_n,$$
 and by hypothesis, we also have 
 $$A_2^*h_5=\lim_{n\rightarrow +\infty}  A_1^* y_n=\lim_{n\rightarrow +\infty}  C_1^*B_1^* y_n=C_1^*B_2^*h_5.$$
 So we have $h_5\in \mathcal N(A_2^*-C_1^*B_2^*) $,  as needed.
 We deduce that 
  $${\mathcal R(A_2-B_2C_1)}\subset\overline{\mathcal R (\tilde B_2)}.$$ 
Hence, for any $\varepsilon>0$ and any $h_2\in H_2$, we have  $||A_2h_1-B_2C_1h_1-B_2u||\leqslant \varepsilon$ for some  $u\in \mathcal N(B_1)$.  Setting $h_2=C_1h_1+u$, we have (i) has needed, since $B_1h_2=B_1(C_1h_1+u)=A_1h_1$.

\end{proof}
Now that we have proved Theorem \ref{th:main}, we can go back to the counterexample evoked in Remark \ref{rerem}.
\begin{proposition}\label{remex}
Condition (i) or (ii) in Proposition \ref{nec} are in general not sufficient for Goal \ref{goal} to hold. 
\end{proposition}
\begin{proof}
 Let  us give an appropriate counterexample.  Consider $H$ any separable Hilbert space of infinite dimension, endowed with a norm $||\cdot||$, and consider $\{e_n\}_{n\in\mathbb N^*}$ a Hilbert basis of $H$. We choose $$H_1=H_2=H_4=H_5=H.$$ 
Assume for the moment that one can find selfadjoint $B_1$ and $B_2$  (so we can forget the adjoint on  these operators in our reasoning) such that  the closure of $(B_2)^{-1}(\mathcal R(B_1))$, denoted by $F$, is not equal to $H$, and such that $(B_2)^{-1}(\overline{\mathcal R(B_1)})=H$.
We take $A_1=0$, so that $A_1=B_1C_1$ with $C_1=0\in \mathcal L_c(H)$. We consider $A_2$ the orthogonal projection on $F^\perp$. Assume that $B_1h_4+B_2h_5=0$ for some $ (h_4,h_5)\in H^2$. We deduce that $B_2 h_5 \in \mathcal R(B_1)$, so $h_5\in F$. By our choice of $A_2$, we deduce that $A_2h_5=0$, so we have the property that $A_1h_4+A_2h_5=0$, as needed. However, we do not have the stronger property that for any $h_1\in H_1$ and any $\varepsilon>0$, there exists  $h_2\in H$ such that $A_1h_1=B_1h_2$ and $||A_2h_1-B_2h_2||\leqslant \varepsilon$. Indeed, by contradiction, this would mean that (iii) of Theorem \ref{th:main} holds. Remark that any $h_5\in H$ is such that there exists some $(y_n)\in H_4^{\mathbb N*}$ for which
 $B_2h_5=\lim_{n\rightarrow +\infty}  B_1 y_n$, since $(B_2)^{-1}(\overline{\mathcal R(B_1)})=H$. We need to conclude that $A_2h_5=0$, \textit{i.e.} $h_5 \in \mathcal N(A_2)=F$, which is of course not possible if we have chosen $h_5\in F^\perp\setminus \{0\}$.

It remains to prove the existence of such $B_1,B_2$. Consider $B_1$ as the unique linear operator such that 
$B_1(e_n)=e_n/n^2$ for all $n\in\mathbb N^*$, and $B_2(h)=\langle h,x\rangle x$, where $x=\sum_{n=1}^{+\infty} e_n/n$.  $B_1$ is a diagonal operator on a Hilbert basis, so it is selfadjoint.  $B_2$ is also clearly selfadjoint by using the definition of the adjoint operator. Moreover, $\mathcal R(B_1)$ is dense in $H$, so that we have $(B_2)^{-1}(\overline{\mathcal R(B_1)})=H$.  Then, 
$h\in  (B_2)^{-1}(\mathcal R(B_1))$ is equivalent to the fact that $B_2(h)\in \mathcal R(B_1)$, which is equivalent to the existence of a sequence $(f_n)_{n\in \mathbb N^*}$ of $l^2(\mathbb N^*)$ such that for any $n\in\mathbb N^*$, we have 
$$\frac{\langle h,x \rangle}{n}= \frac{f_n}{n^2}, \,\,\textit{i.e.} \,\, \langle h,x \rangle= \frac{f_n}{n}.$$
Notably, $f_n/n$ has necessarily to be constant, which means that $f_n=0$, since $(f_n)_{n\in \mathbb N^*}\in l^2(\mathbb N^*)$. We deduce that for any $h\in (B_2)^{-1}(\mathcal R(B_1)$, we have $\langle x,h\rangle = 0$, which implies that 

$(B_2)^{-1}({\mathcal R(B_1)})\subset x^\perp$,  and  concludes our proof since $x^\perp$ is closed.
\end{proof}

Our previous theorem might be difficult to prove in practice. So, our next goal is to give effective sufficient conditions, that can be verified in practice, as we will see afterwards on a simple example. More precisely, in what follows, we are interested to understand what can be added as an hypothesis so that the conditions given in Proposition \ref{nec} become sufficient. The first one is the following.

\begin{proposition}\label{prop1} Assume that $(B_2^{*})^{-1}(\mathcal R(B_1^*))=(B_2^{*})^{-1}(\overline{\mathcal R(B_1^*)})$. Then, condition (i) or (ii) in Proposition \ref{nec} are sufficient in order that Goal \ref{goal} holds.
\end{proposition}
\begin{remark} 
\begin{itemize}
\item Notably, if $\mathcal R(B_1^*)$ (or, equivalently, $\mathcal R(B_1)$) is closed, then, Proposition \ref{prop1} applies. This is automatically the case if one of the Hilbert spaces $H_2$ or $H_4$ is finite-dimensional.
\item Of course, the counterexample provided in Proposition \ref{remex} does not verify $(B_2^{*})^{-1}(\mathcal R(B_1^*))=(B_2^{*})^{-1}(\overline{\mathcal R(B_1^*)})$.
\end{itemize}
\end{remark}

\begin{proof}
Indeed, by (i) in Proposition \ref{nec},  we have that  there exists $C>0$ such that for any $h_4\in H_4$, we have $$||A^*_1h_4||_1\leqslant C ||B^*_1h_4||_2.$$ Now,  consider any $h_5\in H_5$ such that there exists some $(y_n)\in H_4^{\mathbb N*}$ for which
 $B_2^*h_5=\lim_{n\rightarrow +\infty}  B_1^* y_n$. Then,  $h_5 \in   (B_2^{*})^{-1}(\overline{\mathcal R(B_1^*)})$. Since $(B_2^{*})^{-1}(\mathcal R(B_1^*))=(B_2^{*})^{-1}(\overline{\mathcal R(B_1^*)})$,  we deduce that $\lim_{n\rightarrow +\infty}  B_1^* y_n=B_1^*y$ for some $y\in H_4$. So we have 
 $$B_2^*h_5+B_1^*(-y)=0,$$ and we deduce thanks to condition (ii) in Proposition \ref{nec} that  $$A_2^*h_5+A_1^*(-y)=0,$$ \textit{i.e.} $A_2^*h_5=A_1^*y$.
 Since for any $n\in\mathbb N^*$, we have 
 $$||A^*_1(y-y_n)||_1\leqslant C ||B^*(y-y_n)||_2,$$ 
 we deduce that we also have 
    $A_1^*y=\lim_{n\rightarrow +\infty}  A_1^* y_n.$
 So we deduce that 
   $A_2^*h_5=\lim_{n\rightarrow +\infty}  A_1^* y_n$, and Theorem \ref{th:main} applies, which concludes our proof.
\end{proof}

Unfortunately, for the applications we have in mind, $(B_2^{*})^{-1}(\mathcal R(B_1^*))=(B_2^{*})^{-1}(\overline{\mathcal R(B_1^*)})$ will not be verified, or hard to verify, in practice. So, our next goal is to give an effective sufficient condition, that can be verified in practice, as we will see afterwards on a simple example. %We call $F=\overline{\mathcal R(B)}$ and $\Pi_F$ the orthogonal projection on $F$ in $H_4\times H_5$.
Assume that $\mathcal N(B^*)=0$.
Then, introduce the following Hilbertian norm on $H_4$: 
$$||h_4||_*^2= ||B_1^*h_4||_2^2.$$
It is a norm. Indeed, if $B_1^* h_4=0$, notably, we have $B^*(h_4,0)=0$, and so $(h_4,0)=0$. We introduce the completion $\overline{H_4}^{||\cdot||_*}$ for this norm, and 
$$X=\overline{H_4}^{||\cdot||_*}\times H_5,$$
endowed with the Hilbertian norm

$$||(h_4,h_5)||_*^2= ||B_1^*h_4||_2^2+||h_5||_5^2,$$
which makes $X$ a Hilbert space.

Then, by definition of $||\cdot||_*$, one can extend $B^*$ on $X$ as a linear continuous map, still denoted by $B^*$.
Under these hypotheses, we have the following proposition.

\begin{proposition}\label{propp}If $\mathcal N(B^*)=0$, if the extension $B^*$ on $X$ also verifies $\mathcal N(B^*)=0$, and if there  exists $C>0$ such that for any  $ h_4\in H_4$, we have 
$$||A^*_1h_4||_1\leqslant C ||B_1^*h_4||_2,$$ then,  
for  any $h_1\in H_1$ and any $\varepsilon>0$, there exists  $h_2\in H_2$ such that $A_1h_1=B_1h_2$ and $||A_2h_1-B_2h_2||\leqslant \varepsilon$.\end{proposition}

\begin{remark}It might seem surprising the condition of Proposition \ref{propp} does not depend on $A_2$. In fact, this comes from the strong assumption that $\mathcal N(B^*)=0$, which implies by orthogonality that $B$ has dense range. So, $\mathcal R(B)$ is so ``big'' that in this context, the exact form of $A_2$ is not important. This can also be observed during the proof.
\end{remark}
\begin{proof}According to Theorem \ref{th:main}, we only need to prove that for any $h_5\in H_5$ such that there exists some $(y_n)\in H_4^{\mathbb N*}$ for which
 $B_2^*h_5=\lim_{n\rightarrow +\infty}  B_1^* y_n$, then,  $A_2^*h_5=\lim_{n\rightarrow +\infty}  A_1^* y_n.$ Consider such a $h_5$. Then, remark that the sequence $(y_n,0)_{n\in\mathbb N^*}$ is a Cauchy sequence for the $||\cdot||_*$-norm, so it converges to some $(z,0)\in X$. By uniqueness, we necessarily have that $B_1^*z=B_2^*h_5$. So, we have that $(z,-h_5)\in \mathcal N(B^*)$, and hence $z=h_5=0$, which implies that $A_2^*h_5=0$. To conclude, remark that since  $||A^*_1y_n||_1\leqslant C ||B_1^*y_n||_2$ for any $n\in\mathbb N^*$, we also have $A^*_1 y_n\rightarrow 0$ as $n\rightarrow +\infty$, which concludes the proof.
\end{proof}

\section{An example : a coupled heat system on a bounded domain}
\label{app}
Let $T>0$. Let  $\Omega$ a bounded, connected and open set of $\mathbb{R}^d$ $d\in\mathbb N^*$ of class $\mathscr{C}^2$,  $\omega$ a nonempty open subset of $\Omega$.
Consider some reaction-diffusion model of the form 

\begin{equation}
		\label{yz0}
		\left\{ 
			\begin{array}{rcll}
				\dfrac{\partial y}{\partial t}-\Delta y &=&\mathbbm{1}_{\omega
}h& \text{ in }(0,T) \times \Omega,
				\\ 
				y &=& 0 & \text{ on }(0,T) \times \partial \Omega,
				\\ 
				y\left( 0,\cdot\right) &=& y^0 & \text{ in }\Omega,
\\
				\dfrac{\partial z}{\partial t}-\Delta z&=&y & \text{ in }(0,T) \times \Omega,
				\\ 
				z &=& 0 & \text{ on }(0,T) \times \partial \Omega,
				\\ 
				z\left( 0,\cdot\right) &=& z^0 & \text{ in }\Omega,
			\end{array}
		\right.
	\end{equation}
for some $y^0,z^0\in L^2(\Omega)$ and some $h\in L^2((0,T),\Omega)$. Remind that this system is
\begin{itemize}
\item \emph{null-controllable} in arbitrary small time : for any $T>0$, for any $y^0,z^0\in L^2(\Omega)$, there exists $h\in L^2((0,T),\Omega)$ such that $y(T)=z(T)=0$;
\item  and also \emph{approximately controllable} in arbitrary small time: for any $T>0$, for any $y^0,z^0,y^T,z^T\in L^2(\Omega)$, for any $\varepsilon>0$ there exists $h\in L^2((0,T),\Omega)$ such that $||y(T)-y^T||\leqslant \varepsilon$ and $||z(T)-z^T||\leqslant \varepsilon$.
\end{itemize}
 Both results are well-known and consequences for instance of the results given in  \cite{MR2226005}.

As explained in Remark \ref{yz0yz}, for the question we would like to investigate in what follows, there is no loss of generality by assuming that $z^0=0$, so we will instead look at

\begin{equation}
		\label{yz}
		\left\{ 
			\begin{array}{rcll}
				\dfrac{\partial y}{\partial t}-\Delta y &=&\mathbbm{1}_{\omega
}h& \text{ in }(0,T) \times \Omega,
				\\ 
				y &=& 0 & \text{ on }(0,T) \times \partial \Omega,
				\\ 
				y\left( 0,\cdot\right) &=& y^0 & \text{ in }\Omega,
\\
				\dfrac{\partial z}{\partial t}-\Delta z&=&y & \text{ in }(0,T) \times \Omega,
				\\ 
				z &=& 0 & \text{ on }(0,T) \times \partial \Omega,
				\\ 
				z\left( 0,\cdot\right) &=&0& \text{ in }\Omega,
			\end{array}
		\right.
	\end{equation}
which of course enjoys similar approximate and null controllability properties.

Let us introduce the following spaces and operators, according to our previous notations.
\begin{itemize}
\item $H_1=H_2=H_4=H_5=L^2(\Omega)$.
\item $A(y^0)=(y(T),z(T))$, where $(y,z)$ is the corresponding solution of \eqref{yz} for $h=0$ and $z_0=0$, so $A_1(y^0)=y(T)$ and $A_2(y^0)=z(T)$.
\item $B(h)=-(y(T),z(T))$, where $(y,z)$ is the corresponding solution of \eqref{yz} for $y^0=z^0=0$, so $B_1(h)=-y(T)$ and $B_2(h)=-z(T)$. %Notably, we observe that $B_2(h)=0$.
\end{itemize}

Let us now investigate a problem mixing null-controllability and approximate controllability. As far as we know, this result is new.

\begin{proposition}\label{cc}
For any $(y^0,z^T)\in L^2(\Omega)^2$ and any $\varepsilon>0$, there exists $h\in H_2$ such that the solution $(y,z)$ of \eqref{yz} verifies $y(T)=0$ and $||z(T)-z^T||_{L^2(\Omega)}\leqslant \varepsilon$.
\end{proposition}
\begin{proof}
By the superposition principle for linear equations, we have 
$y(T)=A_1(y^0)-B_1(h)$ and $z(T)=A_2(y^0)$. 
So we can reformulate our question as follows : for any $(y^0,z^T)\in \in L^2(\Omega)^2$ and any $\varepsilon>0$, there exists $h\in H_2$ such that the solution $(y,z)$ of \eqref{yz} verifies $A_1(y^0)=B_1(h)$ and $||A_2(y^0)-B_2(h)-z_T||\leqslant \varepsilon$. 

We first compute the adjoint operators of $A_1,B_1$ and $A,B$ (from which we can easily deduce the expression of $A_2^*$ and $B_2^*$, but they are not explicitly needed). It will be convenient to introduce the adjoint system of \eqref{yz}: 

\begin{equation}
		\label{yza}
		\left\{ 
			\begin{array}{rcll}
				-\dfrac{\partial \phi}{\partial t}-\Delta \phi &=&\psi& \text{ in }(0,T) \times \Omega,
				\\ 
				\phi &=& 0 & \text{ on }(0,T) \times \partial \Omega,
				\\ 
				\phi\left(T,\cdot\right) &=& \phi^T & \text{ in }\Omega,
\\
				-\dfrac{\partial \psi}{\partial t}-\Delta \psi&=&0 & \text{ in }(0,T) \times \Omega,
				\\ 
				\psi &=& 0 & \text{ on }(0,T) \times \partial \Omega,
				\\ 
				\psi\left(T,\cdot\right) &=& \psi^T & \text{ in }\Omega,
			\end{array}
		\right.
	\end{equation}
where $\phi^T,\psi^T\in L^2(\Omega)$. 	
We will also need the following scalar equation: 
\begin{equation}
		\label{yza2}
		\left\{ 
			\begin{array}{rcll}
				-\dfrac{\partial \eta}{\partial t}-\Delta \phi &=&0& \text{ in }(0,T) \times \Omega,
				\\ 
				\eta &=& 0 & \text{ on }(0,T) \times \partial \Omega,
				\\ 
				\eta\left(T,\cdot\right) &=& \phi^T & \text{ in }\Omega,
			\end{array}
		\right.
	\end{equation}	
where $\phi^T\in L^2(\Omega)$. Then, reasoning by density and performing integrations by parts, the solution of \eqref{yz} verifies : for any $\phi^T,\psi^T\in L^2(\Omega)$, 
$$\langle y(T),\phi^T\rangle_{L^2(\Omega)}+\langle z(T),\psi^T\rangle_{L^2(\Omega)}-\langle y^0,\phi(0)\rangle_{L^2(\Omega)}=\int_0^T\int_\omega h(t,x)\phi(t,x)dxdt,$$
where $(\phi,\psi)$ verifies \eqref{yza}
Notably, taking $\psi^T=0$ and $h=0$ leads to 
$$\langle y(T),\phi^T\rangle_{L^2(\Omega)}=\langle y^0,\phi(0)\rangle_{L^2(\Omega)}.$$
But $\psi^T=0$ implies that $\psi=0$, which implies that in fact this $\phi$ is equal to the solution $\eta$ of \eqref{yza2}.
Hence, we have 
$$A_1^*(\psi^T)=\eta(0).$$
Now, taking $h=0$ leads to 
$$\begin{aligned}\langle z(T),\psi^T\rangle_{L^2(\Omega)}+\langle y(T),\phi^T\rangle_{L^2(\Omega)}=\langle y^0,\phi(0)\rangle_{L^2(\Omega)}\end{aligned}.$$
We deduce that 
$$A^*(\phi^T,\psi^T)=\phi(0).$$

Applying a similar strategy leads to   
  $$B_1^*(\phi^T,\psi^T)= -\eta\mathbbm 1_\omega$$
  and
$$B^*(\phi^T,\psi^T)= -\phi \mathbbm 1_\omega.$$
Now, our goal is to apply Proposition \ref{propp}. $\mathcal N(B^*)=0$ is true: if $B^*(\phi^T,\psi^T)=0$, then $\phi=0$ on $(0,T)\times\mathbbm 1_\omega$. Differentiating and using \eqref{yza}, we obtain that 
$-\dfrac{\partial \phi}{\partial t}-\Delta \phi =\psi=0$ on $(0,T)\times\omega$. By Holmgren's uniqueness Theorem, $\psi$ is analytic in space, so $\psi=0$ on $(0,T)\times\Omega$. We notably deduce by  continuity in time   that $\psi^T=0$. Going back to the first equation in \eqref{yza} and doing exactly the same reasoning implies that $\phi^T=0$,  as needed.

  Introduce the following Hilbertian norm on $L^2(\Omega)^2$: 
$$||(\phi^T,\psi^T)||_*^2= \int_0^T\int_\omega\eta(t,x)^2dxdt+\int_\Omega \psi^T(x)^2dx,$$
and consider $B^*$ to be the extension of   $B^*$ on the completed space $X$. Then, we still have $\mathcal N(B^*)=0$. In order to prove that, we will prove that for any $(\phi^T,\psi^T)\in X$, we can associate a solution $(\phi,\psi)$ to \eqref{yza} such that $\psi\in L^2((0,T'),L^2(\Omega))$ for any $T'<T$. Indeed, classical Carleman estimates for the heat equation (see \textit{e.g.} \cite[Theorem 9.4.1]{fursikov1996controllability}) imply that for any $(\phi^T,\psi^T)\in L^2(\Omega)^2$, we have, for some $C>0$ (that might depend on $T$ and change from line to line in the following reasoning),
$$\int_{0}^T\int_\Omega e^{-\frac{C}{T-t}} \phi(t,x)^2dxdt\leqslant C\left(\int_0^T\int_\omega \phi(t,x)^2dxdt+\int_0^T\int_\Omega \psi(t,x)^2dxdt \right).$$
Moreover, a classical  well-posedness result also ensures that 
\begin{equation}\label{finn1}\int_0^T\int_\Omega \psi(t,x)^2dxdt \leqslant C \int_\Omega \psi^T(x)^2dx\leqslant C||(\phi^T,\psi^T)||^2_*.\end{equation}
We deduce that 
\begin{equation}\label{finn}\begin{aligned}\int_{0}^T\int_\Omega e^{-\frac{C}{T-t}} \phi(t,x)^2dxdt&\leqslant C\left(\int_0^T\int_\omega \phi(t,x)^2dxdt+ \int_\Omega \psi^T(x)^2dx\right)\\&\leqslant C\left(\int_0^T\int_\omega \eta(t,x)^2dxdt+\int_0^T\int_\omega (\phi-\eta)(t,x)^2dxdt+||(\phi^T,\psi^T)||^2_*\right)\\&\leqslant C\left(\int_0^T\int_\omega (\phi-\eta)(t,x)^2dxdt+||(\phi^T,\psi^T)||^2_*\right).\end{aligned}\end{equation}
Remark that $\tilde \eta = \phi-\eta$ verifies 
\begin{equation}
		\label{yza3}
		\left\{ 
			\begin{array}{rcll}
				-\dfrac{\partial \tilde\eta}{\partial t}-\Delta \tilde\eta &=&\psi& \text{ in }(0,T) \times \Omega,
				\\ 
				\tilde\eta &=& 0 & \text{ on }(0,T) \times \partial \Omega,
				\\ 
				\tilde\eta\left(T,\cdot\right) &=&0& \text{ in }\Omega,
			\end{array}
		\right.
	\end{equation}	
So by a classical well-posedness result, we also have 
$$\begin{aligned}\int_0^T\int_\omega (\psi-\eta)(t,x)^2dxdt &\leqslant C \int_0^T\int_\Omega (\phi-\eta)(t,x)^2dxdt \\&\leqslant C\int_0^T\int_\Omega \psi(t,x)^2dxdt\\&\leqslant C\int_\Omega\psi^T(x)^2dx\\&\leqslant C||(\phi^T,\psi^T)||^2_*.\end{aligned}$$
Going back to \eqref{finn}, we deduce that 

\begin{equation}\label{finnZ}\int_{0}^T\int_\Omega e^{-\frac{C}{T-t}} \phi(t,x)^2dxdt\leqslant C||(\phi^T,\psi^T)||^2_*.\end{equation}

From \eqref{finn1} and \eqref{finnZ}, we deduce by density that for any $T'<T$ and any $(\phi^T,\psi^T)\in X$, we can associate a solution $(\phi,\psi)$ to \eqref{yza} such that for any $T'<T$, we have $y\in L^2((0,T')\times \Omega)$. So notably, if $B^*(\phi^T,\psi^T)=0$, then, we have $\phi=\psi=0$ on $L^2((0,T')\times\Omega)$. This is true for any $T'<T$, so we deduce that $\phi=\psi=0$ on $L^2((0,T)\times\Omega)$. This exactly means that $\mathcal N(B^*)=0$ on $X$, so that Proposition \ref{propp} applies and our result is proved.
\end{proof}
\begin{remark}\label{yz0yz}In fact, Proposition \ref{cc} also implies the same property for system \eqref{yz0}, namely, For any $(y^0,z^0,z^T)\in z^T$ and any $\varepsilon>0$, there exists $h\in H_2$ such that the solution $(y,z)$ of \eqref{yz0} verifies $y(T)=0$ and $||z(T)-z^T||\leqslant \varepsilon$. Indeed, introducing $A_3(z^0)$ the solution of \eqref{yz0} with $y^0=h=0$, and with the notations of the proof of Proposition \ref{cc}, the superposition principle says that our problem is equivalent to: for any $(y^0,z^0,z^T)\in L^2(\Omega)^3$ and any $\varepsilon>0$, there exists $h\in H_2$ such that $A_1(y^0)=B_1(h)$ and $||A_2(y^0)+A_3(z^0)-B_2(h)-z_T||\leqslant \varepsilon$, which is clearly implies by Proposition \ref{cc} (just change $z_T$ into $z_T-A_3(z_0)$).

\end{remark}
\section{Further results, comments, open problems}
\paragraph{More products.} Of course, one can consider operators with codomain in products of more than two spaces: assume that we have some Hilbert spaces $E_1,\ldots E_p$ on which one wants the ``exact range inclusion property'', and $F_1,\ldots F_m$ on which one wants the ``closure of range inclusion''  property. One can obtain the analogue of Theorem \ref{th:main}  just by setting $H_4=E_1\times \ldots \times E_p$ and $H_5=F_1\times\ldots\times F_m$, endowed with the natural Hilbertian product norm.

\paragraph{The case of Banach spaces.} The analog of Douglas' majorization and factorization theorem  in Banach spaces might be false if we work with non-reflexive Banach spaces (see \cite{MR312287}).  However, the result given in \cite{MR312287} notably gives that we have an analog of Douglas' majorization and factorization theorem for reflexive Banach spaces, for which all orthogonality relation that are true in Hilbert spaces also hold. So Theorem \ref{th:main} is also valid in the case of reflexive Banach spaces.

\paragraph{Extending Proposition \ref{propp}.} Even if the condition $\mathcal N(B^*)=0$ is often verified in practice in many examples (notably coming from control theory), it is a very restrictive condition. Here is a possible way to extend a little bit the results. Assume that $\overline{\mathcal R(B)}$ is itself a cartesian subspace of $H_4\times H_5$. Then, we can replace $B$ with $\tilde B$, obtained by restricting the codomain to $\overline{\mathcal R(B)}$. $\tilde B$ is now with dense range, so $\mathcal N(\tilde B^*)=0$. Since $\overline{\mathcal R(B)}$ is cartesian, we can apply Proposition \ref{propp} to this operator.
In the general case, we have no idea of how to find an analogue to Proposition \ref{propp}. Notably, it would be reasonable that a condition on $A_2$ appears in such an analogue.

\paragraph{Unbounded operators.} We do not  treat the general case of unbounded operators $A$ and $B$ on Hilbert or Banach spaces, that  is more involved (see \cite{MR203464} or \cite{MR3191859}). Here, the situation is likely to be even more complicated and we leave it for further investigation.

\section*{Acknowledgements}
This work was funded by the french Agence Nationale de la Recherche (Grant ANR-22-CPJ1-0027-01).

\vspace*{0.2cm}

\noindent The author would like to thank Lucas Davron for useful comments and corrections.

\bibliographystyle{plain}
\bibliography{biblio}
\end{document}